\documentclass[reqno, 12pt]{amsart}
\usepackage{enumerate}
\usepackage{amssymb,url,color}
\usepackage{hyperref}
\usepackage{amsmath,amsthm}

\newtheorem{thm}{Theorem}[section]
\newtheorem{cor}{Corollary}[section]

\newtheorem{lem}{Lemma}[section]
\newtheorem{rem}{Remark}[section]

\theoremstyle{Problem}

\theoremstyle{definition}

\numberwithin{equation}{section}
\newcommand{\pp}{\mathbb{P}}
\newcommand{\ee}{\mathbb{E}}

\newcommand{\FF}{\mathcal{F}}

\newcommand{\rr}{\mathbb{R}}

\def\beq{\begin{equation}}
\def\deq{\end{equation}}
\def\dsp{\displaystyle}
\allowdisplaybreaks 

\bfseries\rmfamily 

\begin{document}
\title[Cram\'{e}r's moderate deviations of modularity in network]
{Cram\'{e}r's moderate deviations of modularity in network}
\thanks{This work is supported by National Natural Science Foundation of China (NSFC-11971154).}

\author[Y. Miao]{Yu Miao}
\address[Y. Miao]{College of Mathematics and Information Science, Henan Normal University, Henan Province, 453007, China; Henan Engineering Laboratory for Big Data Statistical Analysis and Optimal Control, Henan Normal University, Henan Province, 453007, China.} \email{\href{mailto: Y. Miao
<yumiao728@gmail.com>}{yumiao728@gmail.com}; \href{mailto: Y. Miao <yumiao728@126.com>}{yumiao728@126.com}}

\author[Q. Yin]{Qing Yin}
\address[Q. Yin]{College of Mathematics and Information Science, Henan Normal University, Henan Province, 453007, China.}
\email{\href{mailto: Q. Yin
<qingyin1282@163.com>}{qingyin1282@163.com}}

\begin{abstract} Complex networks play a crucial role in understanding physical, biological, social and technological systems. One of the most relevant features of graphs representing real systems is community structure. In this paper, for a specific partition of a given network, we prove the Cram\'{e}r's moderate deviations of modularity for the partition when the size of the network gets large.
\end{abstract}

\keywords{Modularity; Cram\'{e}r's moderate deviations; Network.}

\subjclass[2020]{05C82, 60F05}
\maketitle

\section{Introduction}

The modern science of networks has brought significant advances to our understanding of complex systems. Existing networks often display a high level of local inhomogeneity, with high edge density within certain groups of nodes and low edge density between these groups.  A relevant feature of networks is community structure. Detecting communities is of great importance in understanding, analyzing, and organizing networks, as well as in making informed decisions (see \cite{A-B,Jackson,Newman,N-10}).

Many approaches have been proposed for detecting community structure in networks. For example, Fortunato \cite{Fortunato} provided some striking examples of real networks with community structure. In order to distinguish meaningful structural changes from random fluctuations, Rosvall and Bergstrom \cite{R-B} presented a solution to this problem by using bootstrap resampling accompanied by significance clustering. Lancichinetti et al. \cite{L-F-R} described a measure aimed at quantifying the statistical significance of single communities. Zhang and Chen \cite{Z-C} introduced a statistical framework for modularity based network community detection. Under this framework, a hypothesis testing procedure is developed to determine the significance of an identified community structure. Ma and Barnett \cite{M-B} proved that the largest eigenvalue and modularity are asymptotically uncorrelated, which suggests the need for inference directly on modularity itself when the network is large.

Li and Qi \cite{L-Q} proposed a way of evaluating the significance of any given partition by considering whether this particular partition can arise simply from randomness under the assumption that there is no underlying community structure in the network. They further described the global null hypothesis called free labeling. Under this null hypothesis, they derived the asymptotic distribution of modularity. Moreover, they performed a simulation study to validate the asymptotic behavior and further used some well-known real network data for illustration. The significance of the partition is defined based on this asymptotic distribution, which can help assess its goodness. Two different partitions can also be compared statistically. Simulation studies and real data analyses are performed for illustration. The model for a specific partition of a given network is as follows.

We consider an undirected graph $G$ consisting of $n$ vertices $\{v_1, v_2, \cdots, v_n\}$ and $m$ edges $\{e_1, e_2, \cdots, e_m\}$. Let $A_{ij}$ be the number of edges between vertex $v_i$ and vertex $v_j$, for $1\le i, j\le n$, and $k_i(n)$ denote the degree of vertex $v_i$, which is the number of edges connected to vertex $v_i$ (for convenience, we write $k_i$ instead of $k_i(n)$).
In this paper, we discuss a simple graph, for which $A_{ij}$ is $0$ or $1$ if $i\neq j$ and $A_{ii}=0$. Then it is easy to see that
$$
k_i=\sum_{j=1}^{n}A_{ij}=\sum_{j=1}^{n}A_{ji}, \ \ \ \text{for} \ \ \ 1\leq i\leq n
$$
and
$$
\sum_{i=1}^{n}k_i=2m, \ \ \ \ \text{for} \ \ \ 1\le i\le n.
$$
Let $C$ denote a partition of network $G$ (using the existing community detection method, see Fortunato \cite{Fortunato}), i.e., each vertex $v_i\ (1\leq i\leq n)$ is associated with a group label or color $c_i\in \{1, 2, \cdots, K\}$, where $K$ is the total number of communities by the partition, and we denote $C=(c_1, c_2, \cdots, c_n)$.
Newman \cite{N-06} introduced the following modularity of the partition $C$,
\beq\label{1-1}
Q_n(C)=\frac{1}{2m}\sum_{i,j}\left(A_{ij}-\frac{k_ik_j}{2m}\right)\delta_{c_{i},c_{j}}
=\frac{1}{2m}\sum_{i,j}B_{ij}\delta_{c_{i},c_{j}},
\deq
where
$$
\delta_{c_{i},c_{j}}=
 \begin{cases}  1 \ \ \ \ \ \ &\ \text{if}\ \ c_i=c_j\\
  0 \ \ \ \ \ \ &\ \text{otherwise}
  \end{cases}\\
$$
and
\beq\label{h1}
B_{ij}=A_{ij}-\frac{k_ik_j}{2m},\ \ \ \ 1\le i,j\le n.
\deq
It is not difficult to check that $-1<Q_n(C)<1$, and $Q_n(C)$ is the weighted sum of $B_{ij}$ over all pairs of vertices $i$, $j$ that fall in the same groups. It measures the extent to which vertices of the same type are connected to each other in a network.

For a given partition $C$ of the network, we are interested in whether this partition could be obtained by randomly assigning colors to the vertices. The global null hypothesis $H_0$ is that the colors are assigned to vertices randomly, regardless of the structure of the network. The probability that a given vertex is labeled as group $1$ is $p_1=|Col(1)|/n$, where $Col(1)$ is the cardinality of the set of vertices with color $1$; the probability is $p_2=|Col(2)|/n$ for group $2$, and so on. For any $1\leq k\leq K$, it is easy to check that
$$
p_1+p_2+\cdots+p_K=1, \ \ \ \ p_k\ge0.
$$
The labeling of different vertices is assumed to be independent so $H_0$ is also called {\em free labeling}.

Assume that the partition $C=(c_1,c_2,\ldots,c_n)$ is a random vector, where $c_1,c_2,\ldots,c_n$ are independent identically distributed random variables, and have the following distribution
$$
\pp(c_1=j)=p_j, \ \ \ \ 1\leq j\leq K.
$$
Denote
\beq\label{e-e}
p_{(l)}=\sum_{k=1}^{K}p_k^l, \ \ \ \text{for} \ \ \ l=1, 2, \cdots
\deq
and
\beq\label{b-1}
\bar{h}(c_i,c_j)=\delta_{c_i,c_j}-p_{c_{i}}-p_{c_{j}}+p_{(2)}, \ \ \ 1\le i\neq j\le n.
\deq
In this case, we denote $Q_n(C)$ by $Q_n$ to avoid confusion. Li and Qi \cite{L-Q} proved the following asymptotic normality of $Q_n$ under some conditions:
 \begin{align}\label{clt}
 \frac{Q_n-\mu_n}{\sigma_n}\xrightarrow{d}N(0,1),
 \end{align}
 where $\mu_n$ and $\sigma_n^2$ are given by
\beq\label{2-3}
\mu_n=\ee[Q_n]=-\frac{1-p_{(2)}}{4m^2}\sum_{i=1}^nk_i^2,
\deq
\beq\label{2-3-1}
\sigma_n^2=Var(Q_n)=\frac{p_{(2)}+p_{(2)}^2-2p_{(3)}}{2m^2}\sum_{1\leq i\neq j\leq n}B_{ij}^2+\frac{p_{(3)}-p_{(2)}^2}{m^2}\sum_{i=1}^nB_{ii}^2.
\deq
Yin et al. \cite{M-Y-W-Y} proved the moderate deviation principle of the modularity estimator for the specific partition of a given network. Miao and Yin \cite{M-Y} studied the Berry-Esseen bound and strong law of large numbers of modularity in network when the size of the network gets large.

Let us recall the development of the Cram\'{e}r moderate deviations as follows.
Let $(\eta_i)_{i\ge 1}$ be a sequence of independent and identically distributed centered real random variables satisfying Cram\'{e}r's condition:
\beq\label{22}
\ee\exp\left\{c_0|\eta_1|\right\}<\infty,
\deq
where $c_0$ is a positive constant. Denote $\ee\eta_1^2=\sigma^2$ and $S_n=\sum_{i=1}^n\eta_i$. Cram\'{e}r \cite{Cramer} established an asymptotic expansion of the probabilities of moderate deviations for the partial sums, that is, for all $0<x=o(n^{1/2})$,
\beq\label{22-1}
\Bigg|\ln\frac{\pp\left(S_n/(\sigma\sqrt{n})>x\right)}{1-\Phi(x)}\Bigg|=O\left(\frac{1+x^3}{\sqrt{n}}\right),\ \ \ \text{as}\ \ \ n\to\infty,
\deq
where $\Phi(x)$ is the standard normal distribution function. Cram\'{e}r's moderate deviations for sums of independent random variables have been studied by many authors, (see, for instance, Feller \cite{Feller}, Petrov \cite{Petrov-54} and \cite{Petrov-75}, Sakhanenko \cite{Sakhanenko}, Saulis and Statulevi\v{c}ius \cite{S-S} and
Statulevi\v{c}ius \cite{Statulevicius}). Grama and Haeusler \cite{G-H} developed a new approach for
proving large deviation results for martingales based on a
change of probability measure. It extends to the case of martingales the conjugate distribution
technique due to Cram\'er.
 Fan et al. \cite{F-G-L-2013} gave an expansion of large deviation probabilities for martingales, which extends the classical result due to Cram\'er to the case of martingale differences satisfying the conditional Bernstein condition.
Fan et al. \cite{F-G-L-2020} proved a Cram\'er moderate deviation expansion for martingales with differences having finite conditional moments of order $2+\rho$, $\rho\in(0, 1]$, and finite one-sided
conditional exponential moments.
Fan \cite{Fan} derived Cram\'er-type moderate deviations for stationary sequences of bounded random variables. Fan and Shao \cite{F-S-2022} extended the classical Cram\'er result to the cases of normalized martingales and standardized martingales, with martingale differences satisfying the conditional Bernstein condition.

Based on the above discussions, the main purpose of this paper is to establish Cram\'{e}r's moderate deviations of modularity in network by using the martingale approximation and Cram\'{e}r's moderate deviations for martingales from Fan \cite{Fan} and Fan and Shao \cite{F-S-2022}.
 The paper is organized as follows, our main results are stated and discussed in Section 2. In Section 3, the preliminary lemmas are stated. Proofs of main results are obtained in Section 4. Throughout the paper, the symbol $M$ denotes a positive constant which is not necessarily the same one in each appearance.

\section{Main results}
In order to obtain the main results, we need to introduce the following symbols:
\beq\label{q1}
\delta_n=\left(\frac{p_{(2)}+p_{(2)}^2-2p_{(3)}}{m}\right)^{1/2},
\deq
\beq\label{q21}
\varepsilon_n=\frac{2\max_{1\le i\le n}k_i}{\sqrt{m(p_{(2)}+p_{(2)}^2-2p_{(3)})}},
\deq
\beq\label{q22}
\eta_n=\frac{}{}\frac{\sqrt{64e\max_{1\le i\le n}k_i}}{\sqrt{m^{1/2}(p_{(2)}+p_{(2)}^2-2p_{(3)})}}=\frac{\sqrt{32e}}{\sqrt{p_{(2)}+p_{(2)}^2-2p_{(3)}}}\sqrt{\varepsilon_n}
\deq
 and
\beq\label{q3}
\gamma_n=\frac{4e\max_{1\le i\le n}k_i}{\sqrt{m(p_{(2)}+p_{(2)}^2-2p_{(3)})}}=4e\varepsilon_n.
\deq

\begin{thm}\label{thm2-1}
Assume that the degree sequence $\{k_i, 1\le i\le n\}$ satisfy the following conditions: $\varepsilon_n=o(1)$ as $n\to\infty$,
\beq\label{tt}
\eta_n\le \frac{1}{2} \ \ \text{and} \  \ \varepsilon_n\le \frac{1}{8e} \ \ \ \text{for all}\ n.
\deq
Then it holds that for all $0\le x=o(\eta_n^{-1})$,
\begin{align*}
&\ \ \ \ \Bigg|\ln\frac{\pp\left(\frac{Q_n-\mu_n}{\delta_n}>x\right)}{1-\Phi(x)}\Bigg|\\
&\le M\Big(x^3(\varepsilon_n+\eta_{n})+x^2\gamma_n|\ln\gamma_n|
+(1+x)(\varepsilon_n|\ln\varepsilon_n|+\eta_n|\ln\eta_n|+\gamma_n|\ln\gamma_n|)\Big)\\
&\le M\Big(x^3 \eta_{n}+x^2\eta_n^2|\ln\eta_n|+(1+x)\eta_n|\ln\eta_n|\Big).
\end{align*}
Moreover, the same result holds true when replacing $Q_n-\mu_n$ is replaced by $\mu_n-Q_n$.
\end{thm}

\begin{rem}\label{rem2-1}
For the degree sequence $\{k_i, 1\le i\le n\}$, we assume that there exists a sequence $\{l_n,n\ge 1\}$ such that
$$
d_1l_n\le k_i\le d_2l_n   \ \ \ \text{for} \ \ \ 1\le i\le n,
$$
where
$$
d_1=\frac{(p_{(2)}+p_{(2)}^2-2p_{(3)})^2}{65^4e^2}  \ \ \ \text{and} \ \ \ d_2=\frac{(p_{(2)}+p_{(2)}^2-2p_{(3)})^2}{64^4e^2}.
$$
If $l_n\le (\log n)^5$, then for any $n\ge1$,
$$
\eta_n=\frac{\sqrt{64e\max_{1\le i\le n}k_i}}{\sqrt{m^{1/2}(p_{(2)}+p_{(2)}^2-2p_{(3)})}}
\le \frac{1}{2}n^{-1/4}(\log n)^{5/4}\le\frac{1}{2}.
$$
Similarly, we have $\varepsilon_n\le1/(8e)$. Hence, (\ref{tt}) holds.
\end{rem}

\begin{rem}\label{rem2-2}
Another example is that for every $1\le i\le n$, the degree $k_i$ of vertex $v_i$ can be controlled uniformly by some positive constant, i.e.,
$$
\max_{1\le j\le n}k_j\le \frac{(p_{(2)}+p_{(2)}^2-2p_{(3)})^{1/2}}{16e},
$$
then $\varepsilon_n<1/(8e)$. Moreover, if
$$
\max_{1\le j\le n}k_j\le \frac{p_{(2)}+p_{(2)}^2-2p_{(3)}}{256e},
$$
then $\eta_n\le 1/2$.
\end{rem}

\begin{rem}\label{rem2-3}
For the network $G$ considered in the present paper, since $A_{ii}=0$ and $A_{ij}\in\{0,1\}$ for $i\neq j$, then  we have
$$
0\le m\le n(n-1)/2\ \ \text{and}\ \ 0\le \max\limits_{1\le i\le n}k_i\le n-1.
$$
Notice that, if $\max\limits_{1\le i\le n}k_i=n\sqrt{(p_{(2)}+p_{(2)}^2-2p_{(3)})}/(16e)$, then $\varepsilon_n\ge1/(8e)$. If $\max\limits_{1\le i\le n}k_i=n(p_{(2)}+p_{(2)}^2-2p_{(3)})/(64e)$, then $\eta_n\ge\frac{1}{2}$. Hence, Theorem \ref{thm2-1} does not hold.
\end{rem}

From Theorem \ref{thm2-1}, we have the following result about the equivalence to the normal tail.
\begin{cor}\label{cor2-1}
Under the conditions in Theorem \ref{thm2-1}, it holds that for all $0\le x=o(\eta_n^{-1})$,
\beq\label{2-1-1-1}
\frac{\pp\left(\frac{Q_n-\mu_n}{\delta_n}>x\right)}{1-\Phi(x)}=1+o(1).
\deq
Moreover, the same result holds true when replacing $Q_n-\mu_n$ is replaced by $\mu_n-Q_n$.
\end{cor}

\begin{cor}\label{cor2-2}
Under the conditions in Theorem \ref{thm2-1}, we have
$$
\aligned
\sup_{x\in\rr}\Bigg|\pp\left(\frac{Q_n-\mu_n}{\delta_n}\le x\right)-\Phi(x)\Bigg|\le M\eta_n|\ln\eta_n|.
\endaligned
$$
\end{cor}
\begin{rem}\label{rem2-4}
Assume that the degree sequence $\{k_i, 1\le i\le n\}$ satisfies the following condition:
$$
\frac{\max_{1\le i\le n}k_i}{\sqrt{m}}\le Mn^{-1/2},
$$
then, we derive that
$$
\sup_{x\in\rr}\Bigg|\pp\left(\frac{Q_n-\mu_n}{\delta_n}\le x\right)-\Phi(x)\Bigg|\le Mn^{-1/4}\log n.
$$
\end{rem}

Next we give the Cram\'er moderate deviation for the normalized sums $\dsp\frac{Q_n-\mu_n}{\sigma_n}$.

\begin{thm}\label{thm2-2}
Under the conditions in Theorem \ref{thm2-1}, it holds that for all $0\le x=o(\eta_n^{-1})$,
\begin{align}
\Bigg|\ln\frac{\pp\left(\frac{Q_n-\mu_n}{\sigma_n}>x\right)}{1-\Phi(x)}\Bigg|\le M\Big(x^3 \eta_{n}+x^2\eta_n^2|\ln\eta_n|
+(1+x)\eta_n|\ln\eta_n|\Big)~\label{2-1}
\end{align}
and
\beq\label{ooo-1}
\frac{\pp\left(\frac{Q_n-\mu_n}{\sigma_n}>x\right)}{1-\Phi(x)}=1+o(1).
\deq
Moreover, the same results hold true when replacing $Q_n-\mu_n$ is replaced by $\mu_n-Q_n$.
\end{thm}
\begin{cor}\label{cor2-21}
Under the conditions in Theorem \ref{thm2-2}, we have
$$
\aligned
\sup_{x\in\rr}\Bigg|\pp\left(\frac{Q_n-\mu_n}{\sigma_n}\le x\right)-\Phi(x)\Bigg|\le M\eta_n|\ln\eta_n|.
\endaligned
$$
\end{cor}
\section{Some important lemmas}
Let $(\xi_i,\FF_i)_{i=0,\ldots,n}$ be a finite sequence of martingale differences, defined on a probability space $(\Omega,\FF,\pp)$, where $\xi_0=0,\ \{\emptyset,\Omega\}=\FF_{0}\subseteq\ldots\subseteq\FF_n\subseteq\FF$ are increasing $\sigma$-fields and $(\xi_i)_{i=1,\ldots,n}$ are allowed to depend on $n$. Set
$$
X_0=0,\ \ \ \ \ X_k=\sum_{i=1}^k\xi_i,\ \ \ k=1,\ldots,n.
$$
Then $(X_i,\FF_i)_{i=0,\ldots,n}$ is a martingale. Denote by $\langle X\rangle$ the quadratic characteristic of the martingale $X=(X_k,\FF_k)_{k=0,\ldots,n}$, that is
$$
\langle X\rangle_0=0,\ \ \ \ \ \langle X\rangle_k=\sum_{i=1}^k\ee(\xi_i^2|\FF_{i-1}), \ \ \ k=1,\ldots,n.
$$
In the sequel we shall use the following conditions:
\begin{enumerate}[{\rm(i)}]
\item There exists a number $\varepsilon_n\in(0,\frac{1}{2}]$ such that
   $$
   |\ee\left(\xi_i^k|\FF_{i-1}\right)|\le\frac{1}{2}k!\varepsilon_n^{k-2}\ee\left(\xi_i^2|\FF_{i-1}\right),
   \ \ \ \text{for\ all} \ k\ge 2 \ \text{and}\ 1\le i\le n;
   $$
\item There exist a number $\eta_n\in(0,\frac{1}{2}]$ and a positive constant $M$ such that for all $x>0$,
   $$
   \pp\left(|\langle X\rangle_n-1|\ge x\right)\le M\exp\left\{-x\eta_n^{-2}\right\}.
   $$
\end{enumerate}

In the proof of Theorem \ref{thm2-1}, we need the following Cram\'{e}r's moderate deviations for martingales, which is a simple consequence of Theorem 2.2 of Fan and Shao \cite{F-S-2022}.
\begin{lem}\label{lem2-1}$($\cite[Remark 2.1]{F-S-2022}$)$
Assume that the conditions (i) and (ii) are satisfied. Then for all $0\le x=o(\min\{\varepsilon_n^{-1},\eta_n^{-1}\})$, we have
$$
\Bigg|\ln\frac{\pp(X_n>x)}{1-\Phi(x)}\Bigg|\le M\Big(x^3(\varepsilon_n+\eta_n)+(1+x)\left(\eta_n|\ln\eta_n|+\varepsilon_n|\ln\varepsilon_n|\right)\Big).
$$
\end{lem}

\begin{lem}\label{lem-h-1}$($\cite[(2.16)]{G-L-Z}$)$
If $\xi_1,\xi_2,\ldots,\xi_n$ are centered and independent random variables, then for any $p\ge 2$,
$$
\ee\Bigg|\sum_{i=1}^n\xi_i\Bigg|^p\le 2^p(p-1)^{p/2}\ee\left(\sum_{i=1}^n\xi_i^2\right)^{p/2}.
$$
\end{lem}

\begin{lem}\label{lem-h-2}
Under the conditions in Lemma \ref{lem-h-1}, if there exist positive constants $\{b_{i},1\le i\le n\}$, such that $|\xi_i|\le b_{i}$ for all $1\le i\le n$, then for any $x>0$, we have
\beq\label{0}
\pp\left(\Bigg|\sum_{i=1}^n\xi_i\Bigg|\ge x\right)\le \exp\left\{2-\frac{x}{2eF_n}\right\},
\deq
where
$$
F_n=\left(\sum_{i=1}^nb_i^2\right)^{1/2}.
$$
\end{lem}
\begin{proof}
From Lemma \ref{lem-h-1}, we have
$$
\ee\Bigg|\sum_{i=1}^n\xi_i\Bigg|^p\le 2^p(p-1)^{p/2}\ee\left(\sum_{i=1}^n\xi_i^2\right)^{p/2}\le2^pp^p\left(\sum_{i=1}^nb_i^2\right)^{p/2}=:(2F_n)^pp^p.
$$
Let $p=x/2eF_n$ for any $x>4eF_n$. By using the Markov's inequality, we have
$$
\pp\left(\Bigg|\sum_{i=1}^n\xi_i\Bigg|\ge x\right)\le \exp\left\{-\frac{x}{2eF_n}\right\}.
$$
Moreover, for any $0<x<4eF_n$, we have
$$
\pp\left(\Bigg|\sum_{i=1}^n\xi_i\Bigg|\ge x\right)\le \exp\left\{2-\frac{x}{2eF_n}\right\}.
$$
Hence, for any $x>0$, (\ref{0}) holds.
\end{proof}

\begin{lem}\label{lem-z}$($\cite[(2.18)]{G-L-Z}$)$
Let $X_1, X_2, \cdots, X_n$ be a sequence of independent identically distributed random variables, and for any $1\le i,j\le n$, $h_{i,j}(u,v): \rr^2\to \rr$ be measurable and symmetric with respect to its arguments. For any $1\le i\ne j\le n$, assume that $\ee(h_{i,j}(X_i,X_j)|X_j)=0$, $\ee(h_{i,j}(X_i,X_j)|X_i)=0$ and $\ee|h_{i,j}(X_1,X_2)|^p<\infty$ for some $p\ge 2$. Then we have
$$
\ee\left|\sum_{1\le i< j \le n}h_{i,j}(X_i,X_j)\right|^p\le 4^{p}p^p\ee \left(\sum_{1\le i< j \le n}h_{i,j}^2(X_i,X_j)\right)^{p/2}.
$$
\end{lem}

\begin{lem}\label{lem-e}
Under the conditions in Lemma \ref{lem-z}, if there exist positive constants $\{a_{i,j}, 1\le i,j\le n\}$, such that
$|h_{i,j}(u,v)|\le a_{i,j}$ for all $1\le i,j\le n$ and all $u,v \in\rr$, then for any $x>0$, we have
\beq\label{00}
\pp\left(\left|\sum_{1\le i< j \le n}h_{i,j}(X_i,X_j)\right|>x\right)\le \exp\left\{2-\frac{x}{4eD_n}\right\},
\deq
where
$$
D_n= \left(\sum_{1\le i< j \le n}a_{i,j}^2\right)^{1/2}.
$$
\end{lem}
\begin{proof} From Lemma \ref{lem-z}, we have
$$
\ee\left|\sum_{1\le i< j \le n}h_{i,j}(X_i,X_j)\right|^p\le 4^{p}p^p \left(\sum_{1\le i< j \le n}a_{i,j}^2\right)^{p/2}=:(4D_n)^{p}p^p.
$$
Let $p=x/(4eD_n)$ for any $x>8eD_n$. Then, by the Markov's inequality, we have
$$
\pp\left(\left|\sum_{1\le i< j \le n}h_{i,j}(X_i,X_j)\right|>x\right)\le \exp\left\{-\frac{x}{4eD_n}\right\}.
$$
Moreover, for any $0<x<8eD_n$, we have
$$
\pp\left(\left|\sum_{1\le i< j \le n}h_{i,j}(X_i,X_j)\right|>x\right)\le \exp\left\{2-\frac{x}{4eD_n}\right\}.
$$
Hence, for any $x>0$, (\ref{00}) holds.
\end{proof}

\begin{lem}\label{lem2-3-2} $($\cite[Lemma 4]{L-Q}$)$
Define $o_{ij}=\sum_{l=1}^nA_{il}A_{jl}$ for $1\le i,j\le n$. Then
$$
\sum_{1\le i,j\le n}o_{ij}^2\le\max_{1\le j\le n}k_j\sum_{i=1}^nk_i^2.
$$
\end{lem}

\section{Proofs of main results}
\begin{proof}[\bf Proof of Theorem \ref{thm2-1}]
Since
$$
\sum_{i=1}^nB_{ij}=\sum_{j=1}^nB_{ij}=0,
$$
we can deduce that
$$
\sum_{1\leq i\neq j\leq n}B_{ij}=\sum_{1\leq i,j\leq n}B_{ij}-\sum_{i=1}^nB_{ii}=-\sum_{i=1}^nB_{ii}
$$
and
\begin{align}
Q_n&=\nonumber \frac{1}{2m}\sum_{i=1}^nB_{ii}+\frac{1}{2m}\sum_{1\leq i\neq j\leq n}B_{ij}\bar{h}(c_i,c_j)\\
&\ \ \ \ \nonumber +\frac{1}{2m}\sum_{1\leq i\neq j\leq n}B_{ij}(p_{c_{i}}+p_{c_{j}})-\frac{p_{(2)}}{2m}\sum_{1\leq i\neq j\leq n}B_{ij}\\
&=\nonumber\frac{1+p_{(2)}}{2m}\sum_{i=1}^nB_{ii}+\frac{1}{m}\sum_{1\leq i<j\leq n}B_{ij}\bar{h}(c_i,c_j)-\frac{1}{m}\sum_{i=1}^nB_{ii}p_{c_{i}}\\
&=\frac{1-p_{(2)}}{2m}\sum_{i=1}^nB_{ii}+\frac{1}{m}\sum_{1\leq i<j\leq n}B_{ij}\bar{h}(c_i,c_j)-\frac{1}{m}\sum_{i=1}^nB_{ii}(p_{c_{i}}-p_{(2)}). ~\label{cc}
\end{align}
Combining (\ref{2-3}), (\ref{q1}) and the above equation, it holds that
\begin{align}
\frac{Q_n-\mu_n}{\delta_n}&=\nonumber\frac{1}{m\delta_n}\sum_{1\leq i<j\leq n}B_{ij}\bar{h}(c_i,c_j)-\frac{1}{m\delta_n}\sum_{i=1}^nB_{ii}(p_{c_{i}}-p_{(2)})\\
&=\nonumber\frac{1}{m\delta_n}\sum_{1\leq i<j\leq n}A_{ij}\overline{h}(c_i,c_j)-\frac{1}{2m\delta_n}\sum_{1\leq i<j\leq n}\frac{k_ik_j}{m}\overline{h}(c_i,c_j)\\
&\ \ \ +\frac{1}{2m\delta_n}\sum_{i=1}^n\frac{k_i^2}{m}(p_{c_{i}}-p_{(2)}). ~\label{c-1}
\end{align}
Denote
\beq\label{b-1}
T_n=\frac{1}{m\delta_n}\sum_{j=1}^n\sum_{i=1}^{j-1}A_{ij}\overline{h}(c_i,c_j)=:\frac{1}{m\delta_n}\sum_{j=1}^nz_{nj}.
\deq
Let $\FF_j=\sigma(c_1,c_2,\ldots,c_j)$ denote the $\sigma$-algebra generated by $\{c_1,c_2,\ldots,c_j\}$ for $1\leq j\leq n$. Since $\ee(\bar{h}(c_i,c_j)|\FF_{j-1})=\ee(\bar{h}(c_i,c_j)|c_i)=0$ for any $1\leq i<j\leq n$, it is easy to see that
$$
\ee\left(z_{nj}|\FF_{j-1}\right)=0 \ \ \ \text{for}\ \ \ 2\leq j\leq n.
$$
Therefore, for each $n\geq2$, $\{z_{nj},\ j=2,\ldots,n\}$ forms a martingale difference with respect to $\{\FF_j\}$. Then $(T_i,\FF_i)_{i=0,\ldots,n}$ is a martingale and
$$
\langle T\rangle_n=\frac{1}{m^2\delta_n^2}\sum_{j=1}^n\ee\left(z_{nj}^2|\FF_{j-1}\right).
$$

Next, we will verify that $T_n$ satisfies the conditions (i) and (ii) in Lemma \ref{lem2-1}. By using (\ref{q1}) and the fact $|\overline{h}(c_i,c_j)|\le2$, then for $1\le j\le n$, we can derive
$$
\frac{1}{m\delta_n}|z_{nj}|=\frac{1}{m\delta_n}\bigg|\sum_{i=1}^{j-1}A_{ij}\overline{h}(c_i,c_j)\bigg|\le \frac{2k_j}{m\delta_n}\le \frac{2\max_{1\le j\le n}k_j}{\sqrt{m(p_{(2)}+p_{(2)}^2-2p_{(3)})}}=\varepsilon_n.
$$
Hence, from the condition (\ref{tt}), the condition (i) holds.
Since $A_{ij}\in\{0,1\}$, we have $A_{ij}^2=A_{ij}$. It is enough to get that
\begin{align*}
\frac{1}{m^2\delta_n^2}\sum_{j=1}^n\ee\left(z_{nj}^2|\FF_{j-1}\right)
=& \frac{1}{m^2\delta_n^2}\sum_{j=1}^n
\ee\left(\left(\sum_{i=1}^{j-1}A_{ij}\overline{h}(c_i,c_j)\right)^2\Big|\FF_{j-1}\right)\\
=& \frac{1}{m^2\delta_n^2}\sum_{i=1}^{n-1}\sum_{j=i+1}^{n}
A_{ij}\ee\left(\left(\overline{h}(c_i,c_j)\right)^2|\FF_{j-1}\right)\\
&\ \ + \frac{2}{m^2\delta_n^2}\sum_{1\le i<l\le n-1}\sum_{j=l+1}^nA_{ij}A_{lj}\ee\left(\overline{h}(c_i,c_j)\overline{h}(c_{l},c_j)|\FF_{j-1}\right).
\end{align*}
For every $1\le i< j\le n$, define
\beq\label{l-7}
U_{ni}=\sum_{j=i+1}^{n}A_{ij}\ee\left[\left(\overline{h}(c_i,c_j)\right)^2|\FF_{j-1}\right],
\deq
then $\{U_{ni}, 1\le i\le n-1, n\ge 2\}$ is an array of independent random variables with
$$
|U_{ni}|\le 4\sum_{j=i+1}^{n}A_{ij},\ \ 1\le i\le n-1
$$
and
\begin{align}
\frac{1}{m^2\delta_n^2}\sum_{i=1}^{n-1} \ee U_{ni}
\nonumber=& \ee\left[\frac{1}{m^2\delta_n^2}\sum_{i=1}^{n-1}\sum_{j=i+1}^{n}A_{ij}
\ee\left(\left(\overline{h}(c_i,c_j)\right)^2|\FF_{j-1}\right)\right]\\
 \nonumber=& \frac{1}{m^2\delta_n^2} \sum_{i=1}^{n-1}\sum_{j=i+1}^{n}A_{ij}\ee\left[\ee\left(\left(\overline{h}(c_i,c_j)\right)^2|\FF_{j-1}\right)\right]\\
\nonumber =&\frac{1}{m^2\delta_n^2}\sum_{i=1}^{n-1}\sum_{j=i+1}^{n}A_{ij}
 \ee\left[\ee\left(\left(\delta_{c_i,c_j}-p_{c_{i}}-p_{c_{j}}+p_{(2)}\right)^2|\FF_{j-1}\right)\right]\\
\nonumber =&\frac{1}{m^2\delta_n^2}\sum_{i=1}^{n-1}\sum_{j=i+1}^{n}A_{ij}
 \ee\left[\ee\Big(\left(\delta_{c_i,c_j}-p_{c_{i}}\right)^2+\left(p_{(2)}-p_{c_{j}}\right)^2 \right.\\
\nonumber &\ \ \ \left.
 +2\left(\delta_{c_i,c_j}-p_{c_{i}}\right)\left(p_{(2)}-p_{c_{j}}\Big)|\FF_{j-1}\right)\right]\\
\nonumber =&\frac{1}{m^2\delta_n^2}\sum_{i=1}^{n-1}\sum_{j=i+1}^{n}A_{ij}
 \ee\Big[\ee\Big(\delta_{c_i,c_j}^2+p_{c_{i}}^2-2\delta_{c_i,c_j}p_{c_{i}}+p_{(2)}^2+p_{c_{j}}^2-2p_{(2)}p_{c_{j}} \\
\nonumber &\ \ \
 +2\delta_{c_i,c_j}p_{(2)}-2\delta_{c_i,c_j}p_{c_{j}}-2p_{(2)}p_{c_{i}}+2p_{c_{i}}p_{c_{j}}|\FF_{j-1}\Big)\Big]\\
\nonumber=&\frac{1}{m^2\delta_n^2}\sum_{i=1}^{n-1}\sum_{j=i+1}^{n}A_{ij}\ee\left[ p_{c_{i}}-3p_{c_{i}}^2+2p_{(2)}p_{c_i}-p_{(2)}^2+p_{(3)}\right]\\
\nonumber=&\frac{1}{2m^2\delta_n^2}\left(p_{(2)}+p_{(2)}^2-2p_{(3)}\right)\sum_{1\le i, j\le n}A_{ij}\\
 =&\frac{2m}{2m^2\delta_n^2}\left(p_{(2)}+p_{(2)}^2-2p_{(3)}\right)=1.~\label{l-5}
\end{align}
For $1\le i<l\le n$, denote
\beq\label{l-8}
\overline H_{i,l}(c_i,c_l):=\sum_{j=l+1}^nA_{ij}A_{lj}\ee\left(\overline{h}(c_i,c_j)\overline{h}(c_{l},c_j)|\FF_{j-1}\right),
\deq
then, by using the fact $|\overline{h}(c_i,c_j)|\le2$, we can derive
$$
\left|\overline H_{i,l}(c_i,c_l)\right|\le 4\sum_{j=l+1}^nA_{ij}A_{lj}.
$$
Furthermore, for $1\le i<l<j\le n$, since
\begin{align}
&\ \ \ \ \nonumber\ee\left[\overline{h}(c_i,c_j)\overline{h}(c_{l},c_j)\right]\\
&=\nonumber\ee\left[\ee\left[\left(\delta_{c_i,c_j}-p_{c_i}-p_{c_{j}}+p_{(2)}\right)
\left(\delta_{c_{l},c_j}-p_{c_{l}}-p_{c_{j}}+p_{(2)}\right)|\FF_{j-1}\right]\right]\\
&=\nonumber\ee\left[\ee\Big[\delta_{c_i,c_j}\delta_{c_{l},c_j}-\delta_{c_i,c_j}p_{c_{l}}
-\delta_{c_i,c_j}p_{c_{j}}+\delta_{c_i,c_j}p_{(2)}-\delta_{c_{l},c_j}p_{c_i}+p_{c_i}p_{c_{l}}+p_{c_i}p_{c_{j}} -p_{c_i}p_{(2)} \right.\\
 &\ \ \ \ \ \ \ \ \nonumber\left.  -\delta_{c_{l},c_j}p_{c_{j}}+p_{c_{j}}p_{c_{l}}+p_{c_{j}}^2-p_{c_{j}}p_{(2)}
 +\delta_{c_{l},c_j}p_{(2)}-p_{(2)}p_{c_{l}}-p_{(2)}p_{c_j}+p_{(2)}^2
 |\FF_{j-1}\Big]\right]\\
&=\nonumber\ee\Big[\ee\Big(\delta_{c_i,c_{l}}\frac{p_{c_i}+p_{c_{l}}}{2} -p_{c_i}p_{c_{l}}-p_{c_i}^2+p_{c_i}p_{(2)}-p_{c_i}p_{c_{l}}+p_{c_i}p_{c_{l}}+p_{c_i}p_{(2)}-p_{c_i}p_{(2)} \\
 & \ \ \ \ \ \ \ \ \nonumber \ -p_{c_{l}}^2+p_{c_{l}}p_{(2)}+p_{(3)}-p_{(2)}^2+p_{c_{l}}p_{(2)}
 -p_{c_{l}}p_{(2)}-p_{(2)}^2+p_{(2)}^2|c_i\Big)\Big]\\
&=\ee\left[\ee\left(
\delta_{c_i,c_{l}}\frac{p_{c_i}+p_{c_{l}}}{2}-p_{c_i}p_{c_{l}}-p_{c_i}^2
+p_{c_i}p_{(2)}-p_{c_{l}}^2+p_{c_{l}}p_{(2)} +p_{(3)}-p_{(2)}^2|c_i\right)\right]=0,~\label{l-6}
\end{align}
then we have
\begin{align*}
&\ \ \ \ \frac{2}{m^2\delta_n^2}\sum_{1\le i<l\le n-1}\ee\overline H_{i,l}(c_i,c_l)\\
&=\frac{2}{m^2\delta_n^2}\sum_{1\le i<l\le n-1}\ee\left[\sum_{j=l+1}^nA_{ij}A_{lj}\ee\left(\overline{h}(c_i,c_j)\overline{h}(c_{l},c_j)|\FF_{j-1}\right)\right]\\
&=\frac{2}{m^2\delta_n^2}\sum_{1\le i<l\le n-1}
\sum_{j=l+1}^nA_{ij}A_{lj}\ee\left[\overline{h}(c_i,c_j)\overline{h}(c_{l},c_j)\right]=0.
\end{align*}
Therefore, it follows that
\begin{align}
&\ \ \ \ \nonumber\pp\left(|\langle T\rangle_n-1|\ge x\right)\\
&= \nonumber\pp\left(\left|\frac{1}{m^2\delta_n^2}\sum_{j=1}^n\ee(z_{nj}^2|\FF_{j-1})-1\right|\ge x\right)\\
&\le \nonumber\pp\left(\left|\frac{1}{m^2\delta_n^2}\sum_{i=1}^{n-1}\sum_{j=i+1}^{n}
A_{ij}\ee\left(\left(\overline{h}(c_i,c_j)\right)^2|\FF_{j-1}\right)-1\right|\ge \frac{x}{2}\right)\\
&\ \ \ \ \nonumber +\pp\left(\left|\frac{2}{m^2\delta_n^2}\sum_{1\le i<l\le n-1}\sum_{j=l+1}^nA_{ij}A_{lj}\ee\left(\overline{h}(c_i,c_j)\overline{h}(c_{l},c_j)|\FF_{j-1}\right)\right|\ge \frac{x}{2}\right)\\
&=\pp\left(\left|\frac{1}{m^2\delta_n^2}\sum_{i=1}^{n-1}(U_{ni}-\ee U_{ni})\right|\ge \frac{x}{2}\right)
 +\pp\left(\left|\frac{2}{m^2\delta_n^2}\sum_{1\le i<l\le n-1}\overline H_{i,l}(c_i,c_l)\right|\ge \frac{x}{2}\right).~\label{p-1}
\end{align}
Firstly, consider the first part. By using the Lemma \ref{lem-h-2}, for any $x>0$, we deduce that
\begin{align}
&\ \ \ \ \nonumber\pp\left(\left|\frac{1}{m^2\delta_n^2}\sum_{i=1}^{n-1}(U_{ni}-\ee U_{ni})\right|\ge \frac{x}{2}\right)\\
&\le \nonumber \exp\left(2-\frac{xm^2\delta_n^2}{32e\left(\sum_{i=1}^{n-1}\left(\sum_{j=i+1}^{n}A_{ij}\right)^2\right)^{1/2}}\right)\\
&\le \nonumber \exp\left(2-\frac{xm(p_{(2)}+p_{(2)}^2-2p_{(3)})}{32e\left(\sum_{i=1}^{n-1}\left(\sum_{j=i+1}^{n}A_{ij}\right)^2\right)^{1/2}}\right)\\
&\le \exp\left(2-\frac{x(p_{(2)}+p_{(2)}^2-2p_{(3)})\sqrt{m}}{32e\sqrt{\max_{1\le i\le n}k_i}}\right).~\label{p-2}
\end{align}
Next, combining Lemma \ref{lem-e} and \ref{lem2-3-2}, for any $x>0$, we have
\begin{align}
&\ \ \ \ \nonumber\pp\left(\left|\frac{2}{m^2\delta_n^2}\sum_{1\le i<l\le n-1}\overline H_{i,l}(c_i,c_l)\right|\ge \frac{x}{2}\right)\\
&\le \nonumber\exp\left(2-\frac{xm^2\delta_n^2}{64e\left(\sum_{1\le i<l\le n-1}\left(\sum_{j=l+1}^nA_{ij}A_{lj}\right)^2\right)^{1/2}}\right)\\
&\le \nonumber\exp\left(2-\frac{xm(p_{(2)}+p_{(2)}^2-2p_{(3)})}{64e\left(\sum_{1\le i<l\le n-1}\left(\sum_{j=l+1}^nA_{ij}A_{lj}\right)^2\right)^{1/2}}\right)\\
&\le \exp\left(2-\frac{x(p_{(2)}+p_{(2)}^2-2p_{(3)})\sqrt{m}}{64e\max_{1\le i\le n}k_i}\right).~\label{p-3}
\end{align}
From (\ref{p-1}), (\ref{p-2}) and (\ref{p-3}), we have
$$
\pp\left(|\langle T\rangle_n-1|\ge x\right)\le M\exp\left(-\frac{x(p_{(2)}+p_{(2)}^2-2p_{(3)})\sqrt{m}}{64e\max_{1\le i\le n}k_i}\right).
$$
Let
$$
\eta_n^2=64e\max_{1\le i\le n}k_i/(\sqrt{m}(p_{(2)}+p_{(2)}^2-2p_{(3)})),
$$
then, for all $x>0$, we have
$$
\pp\left(|\langle T\rangle_n-1|\ge x\right)\le M \exp\left\{-x\eta_n^{-2}\right\}.
$$
Hence, from Lemma \ref{lem2-1}, for all $0\le x=o\left(\min\{\varepsilon_n^{-1},\eta_n^{-1}\}\right)=o(\eta_n^{-1})$, we have
\beq\label{dd}
\Bigg|\ln\frac{\pp(T_n>x)}{1-\Phi(x)}\Bigg|\le M\Big(x^3(\varepsilon_n+\eta_n)+(1+x)\left(\eta_n|\ln\eta_n|+\varepsilon_n|\ln\varepsilon_n|\right)\Big).
\deq
By using (\ref{c-1}), it holds that
\begin{align}
&\ \ \ \ \nonumber \pp\left(\frac{Q_n-\mu_n}{\delta_n}>x\right)\\
&=\nonumber\pp\left(\frac{1}{m\delta_n}\left(\sum_{1\le i<j\le n}A_{ij}\bar{h}(c_i,c_j)-\sum_{1\le i<j\le n}\frac{k_ik_j}{2m}\bar{h}(c_i,c_j)+\sum_{i=1}^n\frac{k_i^2}{2m}(p_{c_{i}}-p_{(2)})\right)>x\right)\\
&\le\nonumber\pp\Big(T_n>x(1-\gamma_n|\ln\gamma_n|)\Big)+\pp\left(-\frac{1}{m\delta_n}\sum_{1\le i<j\le n}\frac{k_ik_j}{2m}\bar{h}(c_i,c_j)>\frac{x}{2}\gamma_n|\ln\gamma_n|\right)\\
&\ \ \ +\pp\left(\frac{1}{m\delta_n}\sum_{i=1}^n\frac{k_i^2}{2m}(p_{c_{i}}-p_{(2)})>\frac{x}{2}\gamma_n|\ln\gamma_n|\right) ~\label{c-1-1}
\end{align}
and

$$
\aligned
&\ \ \ \ \pp\left(\frac{Q_n-\mu_n}{\delta_n}>x\right)\\
&=\pp\left(\frac{1}{m\delta_n}\left(\sum_{1\le i<j\le n}A_{ij}\bar{h}(c_i,c_j)-\sum_{1\le i<j\le n}\frac{k_ik_j}{2m}\bar{h}(c_i,c_j)+\sum_{i=1}^n\frac{k_i^2}{2m}(p_{c_{i}}-p_{(2)})\right)>x\right)\\
&\ge\pp\Big(T_n>x(1+\gamma_n|\ln\gamma_n|)\Big)-\pp\left(\frac{1}{m\delta_n}\sum_{1\le i<j\le n}\frac{k_ik_j}{2m}\bar{h}(c_i,c_j)>\frac{x}{2}\gamma_n|\ln\gamma_n|\right)\\
&\ \ \ -\pp\left(-\frac{1}{m\delta_n}\sum_{i=1}^n\frac{k_i^2}{2m}(p_{c_{i}}-p_{(2)})>\frac{x}{2}\gamma_n|\ln\gamma_n|\right).
\endaligned
$$
From Lemma \ref{lem-e} and
\beq\label{oo}
\gamma_n=\frac{4e\max_{1\le i\le n}k_i}{\sqrt{m(p_{(2)}+p_{(2)}^2-2p_{(3)})}},
\deq
we deduce that for any $x>0$,
\begin{align}
&\ \ \ \ \nonumber\pp\left(-\frac{1}{m\delta_n}\sum_{1\le i<j\le
n}\frac{k_ik_j}{2m}\bar{h}(c_i,c_j)>\frac{x}{2}\gamma_n|\ln\gamma_n|\right)\\
&\le\nonumber \pp\left(\bigg|\frac{1}{m\delta_n}\sum_{1\le i<j\le
n}\frac{k_ik_j}{2m}\bar{h}(c_i,c_j)\bigg|>\frac{x}{2}\gamma_n|\ln\gamma_n|\right)\\
&\le \nonumber\exp\left\{2-\frac{m\delta_n\gamma_n|\ln\gamma_n|x}{4e\left(\sum_{1\le i<j\le n}\frac{k_i^2k_j^2}{m^2}\right)^{1/2}}\right\}\\
&\le \nonumber\exp\left\{2-\frac{m\delta_n\gamma_n|\ln\gamma_n|x}{4e\max_{1\le i\le n}k_i}\right\}\\
&= \exp\left\{2-|\ln\gamma_n|x\right\}. ~\label{g-1}
\end{align}
Moreover, for any $x>0$, by using Lemma \ref{lem-h-2} and (\ref{oo}), we have
\begin{align}
&\ \ \ \ \nonumber\pp\left(\frac{1}{m\delta_n}\sum_{i=1}^n
\frac{k_i^2}{2m}(p_{c_{i}}-p_{(2)})>\frac{x}{2}\gamma_n|\ln\gamma_n|\right)\\
&\le \nonumber\pp\left(\bigg|\frac{1}{m\delta_n}\sum_{i=1}^n
\frac{k_i^2}{2m}(p_{c_{i}}-p_{(2)})\bigg|>\frac{x}{2}\gamma_n|\ln\gamma_n|\right)\\
&\le \nonumber \exp\left\{2-\frac{m^2\delta_n\gamma_n|\ln\gamma_n|x}{2e\left(\sum_{i=1}^nk_i^4\right)^{1/2}}\right\}\\
&\le \nonumber \exp\left\{2-\frac{m^{3/2}\delta_n\gamma_n|\ln\gamma_n|x}{2e\max_{1\le i\le n}k_i^{3/2}}\right\}\\
&\le \nonumber \exp\left\{2-\frac{2m^{1/2}|\ln\gamma_n|x}{\max_{1\le i\le n}k_i^{1/2}}\right\}\\
&\le \exp\left\{2-|\ln\gamma_n|x\right\}. ~\label{g-2}
\end{align}
Notice that for all $x\ge0$ and $|\tau|\le \frac{1}{2}$,
\beq\label{d-1}
\frac{1-\Phi(x+\tau)}{1-\Phi(x)}=\exp\left\{\theta\sqrt{2\pi}(1+x)|\tau|\right\}
\deq
and
\beq\label{d-2}
\frac{1}{\sqrt{2\pi}(1+x)}e^{-x^2/2}\le 1-\Phi(x)\le \frac{1}{\sqrt{\pi}(1+x)}e^{-x^2/2},
\deq
where $|\theta|\le1$. From (\ref{dd})-(\ref{d-2}), we deduce that for all $1\le x=o\left(\eta_n^{-1}\right)$,
\begin{align}
&\ \ \ \ \nonumber\frac{\pp\left(\frac{Q_n-\mu_n}{\delta_n}>x\right)}{1-\Phi(x)}\\
&\le \nonumber\frac{\pp\Big(T_n>x(1-\gamma_n|\ln\gamma_n|)\Big)}{1-\Phi(x)}
+\frac{\pp\left(-\frac{1}{m\delta_n}\sum_{1\le i<j\le n}\frac{k_ik_j}{2m}\bar{h}(c_i,c_j)>\frac{x}{2}\gamma_n|\ln\gamma_n|\right)}{1-\Phi(x)}\\
&\ \ \ \ \nonumber +\frac{\pp\left(\frac{1}{m\delta_n}\sum_{i=1}^n\frac{k_i^2}{2m}
(p_{c_{i}}-p_{(2)})>\frac{x}{2}\gamma_n|\ln\gamma_n|\right)}{1-\Phi(x)}\\
&\le \nonumber\frac{\pp\Big(T_n>x(1-\gamma_n|\ln\gamma_n|)\Big)}{1-\Phi\Big(x(1-\gamma_n|\ln\gamma_n|)\Big)}\cdot
\frac{1-\Phi\Big(x(1-\gamma_n|\ln\gamma_n|)\Big)}{1-\Phi(x)}\\
&\ \ \ \ \nonumber+\frac{\pp\left(-\frac{1}{m\delta_n}\sum_{1\le i<j\le n}\frac{k_ik_j}{2m}\bar{h}(c_i,c_j)>\frac{x}{2}\gamma_n|\ln\gamma_n|\right)}{1-\Phi(x)}\\
&\ \ \ \ \nonumber +\frac{\pp\left(\frac{1}{m\delta_n}\sum_{i=1}^n\frac{k_i^2}{2m}
(p_{c_{i}}-p_{(2)})>\frac{x}{2}\gamma_n|\ln\gamma_n|\right)}{1-\Phi(x)}\\
&\le \nonumber\exp\Big\{M\Big(x^3(\varepsilon_n+\eta_n)
+(1+x)(\varepsilon_n|\ln\varepsilon_n|+\eta_n|\ln\eta_n|)\Big)\Big\}
\times\frac{1-\Phi\Big(x(1-\gamma_n|\ln\gamma_n|)\Big)}{1-\Phi(x)}\\
& \ \ \ \ \nonumber+\frac{2e^2}{1-\Phi(x)}\exp\left\{-|\ln\gamma_n|x\right\}\\
&\le \nonumber\exp\Big\{M\Big(x^3(\varepsilon_n+\eta_n)
+(1+x)(\varepsilon_n|\ln\varepsilon_n|+\eta_n|\ln\eta_n|)\Big)\Big\}\\
& \ \ \ \ \nonumber\times\exp\left\{\theta\sqrt{2\pi}(1+x)x\gamma_n|\ln\gamma_n|\right\}
+M(1+x)\exp\{x^2/2\}\exp\left\{-|\ln\gamma_n|x\right\}\\
&\le \nonumber\nonumber\exp\Big\{M\Big(x^3(\varepsilon_n+\eta_n)
+(1+x)(\varepsilon_n|\ln\varepsilon_n|+\eta_n|\ln\eta_n|)\Big)\Big\}
\times\exp\left\{Mx^2\gamma_n|\ln\gamma_n|\right\}\\
&\ \ \ \ \nonumber+M(1+x)\exp\left\{\frac{x^2}{2}-|\ln\gamma_n|x\right\}\\
&\le \exp\Big\{M\Big(x^3(\varepsilon_n+\eta_n)+x^2\gamma_n|\ln\gamma_n|
+(1+x)(\varepsilon_n|\ln\varepsilon_n|+\eta_n|\ln\eta_n|)
\Big)\Big\}, ~\label{d-1-1}
\end{align}
where the last line follows by
$$
(1+x)\exp\left\{\frac{x^2}{2}-|\ln\gamma_n|x\right\}\le M\exp\left\{x^3\varepsilon_n\right\}.
$$
By an argument similar to the proof of the above equation, we deduce that for all
$1\le x=o(\eta_n^{-1})$,
\beq\label{d-1-2}
\frac{\pp\left(\frac{Q_n-\mu_n}{\delta_n}>x\right)}{1-\Phi(x)}
\ge \exp\Big\{-M\Big(x^3(\varepsilon_n+\eta_n)+x^2\gamma_n|\ln\gamma_n|
+(1+x)(\varepsilon_n|\ln\varepsilon_n|+\eta_n|\ln\eta_n|)\Big)\Big\}.
\deq
Hence, combining (\ref{d-1-1}) with (\ref{d-1-2}), we have, for all $1\le x=o(\eta_n^{-1})$,
$$
\aligned
&\ \ \ \ \Bigg|\ln\frac{\pp\left(\frac{Q_n-\mu_n}{\delta_n}>x\right)}{1-\Phi(x)}\Bigg|\\
&\le M\Big(x^3(\varepsilon_n+\eta_n)+x^2\gamma_n|\ln\gamma_n|
+(1+x)(\varepsilon_n|\ln\varepsilon_n|+\eta_n|\ln\eta_n|)\Big)\\
&\le M\Big(x^3 \eta_{n}+x^2\eta_n^2|\ln\eta_n|
+(1+x)\eta_n|\ln\eta_n|\Big)\\
\endaligned
$$

Next, we consider the case where $x\in[0,1]$.
It follows that (\ref{dd}) holds also for $-z_{nj}$. Thus, from (\ref{dd}), (\ref{d-2}) and the inequality $|e^x-1|\le |x|e^{|x|}$, we have
\begin{align*}
&\ \ \ \ \sup_{|x|\le 2}\Big|\pp\left(T_n>x\right)-\left(1-\Phi(x)\right)\Big|\\
&\le\sup_{|x|\le 2}\Big(1-\Phi(x)\Big)\Big|e^{M\left(x^3(\varepsilon_n+\eta_n)
+(1+x)\left(\eta_n|\ln\eta_n|+\varepsilon_n|\ln\varepsilon_n|\right)\right)}-1\Big|\\
&\le \sup_{|x|\le 2}M\Big(1-\Phi(x)\Big)\Big(x^3(\varepsilon_n+\eta_n)
+(1+x)\left(\eta_n|\ln\eta_n|+\varepsilon_n|\ln\varepsilon_n|\right)\Big)\\
&\le M\Big(\varepsilon_n+\eta_n+\eta_n|\ln\eta_n|+\varepsilon_n|\ln\varepsilon_n|\Big)\\
&\le M\Big(\varepsilon_n|\ln\varepsilon_n|+\eta_n|\ln\eta_n|\Big).
\end{align*}
By using (\ref{d-1}), we have
\begin{align*}
&\ \ \ \ \Big|\Big(1-\Phi(x+\gamma_n|\ln\gamma_n|)\Big)-\Big(1-\Phi(x)\Big)\Big|\\
&=\Big(1-\Phi(x)\Big)\Bigg|\frac{\Big(1-\Phi(x+\gamma_n|\ln\gamma_n|)\Big)}{\Big(1-\Phi(x)\Big)}-1\Bigg|\\
&\le \frac{M}{1+x}\exp\{-x^2/2\}\Big|\exp\{M(1+x)\gamma_n|\ln\gamma_n|\}-1\Big|\\
&\le \frac{M}{1+x}\exp\{-x^2/2\}(1+x)\gamma_n|\ln\gamma_n|\exp\{(1+x)\gamma_n|\ln\gamma_n|\}\\
&\le M\gamma_n|\ln\gamma_n|.
\end{align*}
Hence, from the inequality $|e^x-1|\le |x|e^{|x|}$, (\ref{g-1}), (\ref{g-2}) and the above inequality, we deduce for all $x\in[0,1]$,
\begin{align*}
&\ \ \ \ \pp\left(\frac{Q_n-\mu_n}{\delta_n}>x\right)-\Big(1-\Phi(x)\Big)\\
&\ge\pp\Big(T_n>x+\gamma_n|\ln\gamma_n|\Big)-\Big(1-\Phi(x)\Big)-\pp\left(\frac{1}{m\delta_n}\sum_{1\le i<j\le n}\frac{k_ik_j}{2m}\bar{h}(c_i,c_j)>\frac{1}{2}\gamma_n|\ln\gamma_n|\right)\\
&\ \ \ -\pp\left(-\frac{1}{m\delta_n}\sum_{i=1}^n\frac{k_i^2}{2m}(p_{c_{i}}-p_{(2)})>\frac{1}{2}\gamma_n|\ln\gamma_n|\right)
\\
&\ge\pp\Big(T_n>x+\gamma_n|\ln\gamma_n|\Big)-\Big(1-\Phi(x+\gamma_n|\ln\gamma_n|)\Big)\\
&\ \ \ \ -\Big|\Big(1-\Phi(x+\gamma_n|\ln\gamma_n|)\Big)-\Big(1-\Phi(x)\Big)\Big|\\
&\ \ \ \ -\pp\left(\frac{1}{m\delta_n}\sum_{1\le i<j\le n}\frac{k_ik_j}{2m}\bar{h}(c_i,c_j)>\frac{1}{2}\gamma_n|\ln\gamma_n|\right)\\
&\ \ \ \ -\pp\left(-\frac{1}{m\delta_n}\sum_{i=1}^n\frac{k_i^2}{2m}(p_{c_{i}}-p_{(2)})>\frac{1}{2}\gamma_n|\ln\gamma_n|\right)
\\
&\ge -M\Big(\varepsilon_n|\ln\varepsilon_n|+\eta_n|\ln\eta_n|\Big)
-\Big|\Big(1-\Phi(x+\gamma_n|\ln\gamma_n|)\Big)-\Big(1-\Phi(x)\Big)\Big|\\
&\ \ \ \ -\pp\left(\frac{1}{m\delta_n}\sum_{1\le i<j\le n}\frac{k_ik_j}{2m}\bar{h}(c_i,c_j)>\frac{1}{2}\gamma_n|\ln\gamma_n|\right)\\
&\ \ \ \ -\pp\left(-\frac{1}{m\delta_n}\sum_{i=1}^n\frac{k_i^2}{2m}(p_{c_{i}}-p_{(2)})>\frac{1}{2}\gamma_n|\ln\gamma_n|\right)
\\
&\ge -M\Big(\varepsilon_n|\ln\varepsilon_n|+\eta_n|\ln\eta_n|+\gamma_n|\ln\gamma_n|\Big).
\end{align*}
Similarly, we derive for all $x\in[0,1]$,
$$
\pp\left(\frac{Q_n-\mu_n}{\delta_n}\ge x\right)-\Big(1-\Phi(x)\Big)\le M\Big(\varepsilon_n|\ln\varepsilon_n|+\eta_n|\ln\eta_n|+\gamma_n|\ln\gamma_n|\Big).
$$
From (\ref{d-2}), we have
\begin{align*}
\frac{\pp\left(\frac{Q_n-\mu_n}{\delta_n}>x\right)}{1-\Phi(x)}-1
&\le \frac{M}{1-\Phi(x)}\Big(\varepsilon_n|\ln\varepsilon_n|+\eta_n|\ln\eta_n|+\gamma_n|\ln\gamma_n|\Big)\\
&\le M(1+x)\Big(\varepsilon_n|\ln\varepsilon_n|+\eta_n|\ln\eta_n|+\gamma_n|\ln\gamma_n|\Big),
\end{align*}
which implies
\begin{align*}
\frac{\pp\left(\frac{Q_n-\mu_n}{\delta_n}>x\right)}{1-\Phi(x)}
&\le 1+M(1+x)\Big(\varepsilon_n|\ln\varepsilon_n|+\eta_n|\ln\eta_n|+\gamma_n|\ln\gamma_n|\Big)\\
&\le 1+Mx^3(\varepsilon_n+\eta_n)+Mx^2\gamma_n|\ln\gamma_n|\\
&\ \ \ \ +M(1+x)\Big(\varepsilon_n|\ln\varepsilon_n|+\eta_n|\ln\eta_n|+\gamma_n|\ln\gamma_n|\Big).
\end{align*}
Then for all $x\in[0,1]$, it holds that
\begin{align*}
&\ \ \ \ \ln\frac{\pp\left(\frac{Q_n-\mu_n}{\delta_n}>x\right)}{1-\Phi(x)}\\
&\le \ln\Big(1+Mx^3(\varepsilon_n+\eta_n)+Mx^2\gamma_n|\ln\gamma_n|
+M(1+x)(\varepsilon_n|\ln\varepsilon_n|+\eta_n|\ln\eta_n|+\gamma_n|\ln\gamma_n|)\Big)\\
&\le M\Big(x^3(\varepsilon_n+\eta_n)+x^2\gamma_n|\ln\gamma_n|
+(1+x)(\varepsilon_n|\ln\varepsilon_n|+\eta_n|\ln\eta_n|+\gamma_n|\ln\gamma_n|)\Big).
\end{align*}
Similarly, we derive for all $x\in[0,1]$,
\begin{align*}
&\ \ \ \ \ln\frac{\pp\left(\frac{Q_n-\mu_n}{\delta_n}>x\right)}{1-\Phi(x)}\\
&\ge -M\Big(x^3(\varepsilon_n+\eta_n)+x^2\gamma_n|\ln\gamma_n|
+(1+x)(\varepsilon_n|\ln\varepsilon_n|+\eta_n|\ln\eta_n|+\gamma_n|\ln\gamma_n|)\Big).
\end{align*}
Therefore, for all $x\in[0,1]$, we have
\begin{align*}
&\ \ \ \ \Bigg|\ln\frac{\pp\left(\frac{Q_n-\mu_n}{\delta_n}>x\right)}{1-\Phi(x)}\Bigg|\\
&\le M\Big(x^3(\varepsilon_n+\eta_n)+x^2\gamma_n|\ln\gamma_n|
+(1+x)(\varepsilon_n|\ln\varepsilon_n|+\eta_n|\ln\eta_n|+\gamma_n|\ln\gamma_n|)\Big)\\
&\le M\Big(x^3 \eta_{n}+x^2\eta_n^2|\ln\eta_n|+(1+x)\eta_n|\ln\eta_n|\Big),
\end{align*}
which implies the desired result for all $0\le x=o(\eta_n^{-1})$.
\end{proof}

\begin{proof}[\bf Proof of Corollary \ref{cor2-1}]
By using Theorem \ref{thm2-1}, we have
\begin{align*}
\ln\frac{\pp\left(\frac{Q_n-\mu_n}{\delta_n}>x\right)}{1-\Phi(x)}
\le M\Big(x^3 \eta_{n}+x^2\eta_n^2|\ln\eta_n|+(1+x)\eta_n|\ln\eta_n|\Big),
\end{align*}
which implies that
\begin{align*}
\frac{\pp\left(\frac{Q_n-\mu_n}{\delta_n}> x\right)}{1-\Phi(x)}
\le\exp\Big\{M\Big(x^3 \eta_{n}+x^2\eta_n^2|\ln\eta_n|+(1+x)\eta_n|\ln\eta_n|\Big)\Big\}.
\end{align*}
Then (\ref{2-1-1-1}) holds for all $0\le x=o(\eta_n^{-1})$.
\end{proof}

\begin{proof}[\bf Proof of Corollary \ref{cor2-2}]
Let $\kappa_n=\min\{\varepsilon_n^{-1/4},\eta_n^{-1/4},\gamma_n^{-1/4}\}=\eta_n^{-1/4}$, then we have
\begin{align}
&\ \ \ \ \nonumber\sup_{x\in\rr}\Bigg|\pp\left(\frac{Q_n-\mu_n}{\delta_n}\le x\right)-\Phi(x)\Bigg|\\
&\le \nonumber\sup_{|x|\le\kappa_n}\Bigg|\pp\left(\frac{Q_n-\mu_n}{\delta_n}\le x\right)-\Phi(x)\Bigg|
+\sup_{|x|>\kappa_n}\Bigg|\pp\left(\frac{Q_n-\mu_n}{\delta_n}\le x\right)-\Phi(x)\Bigg|\\
&\le \nonumber\sup_{|x|\le \kappa_n}\Bigg|\pp\left(\frac{Q_n-\mu_n}{\delta_n}\le x\right)-\Phi(x)\Bigg|
+\sup_{x<-\kappa_n}\pp\left(\frac{Q_n-\mu_n}{\delta_n}\le x\right)+\sup_{x<-\kappa_n}\Phi(x)\\
&\ \ \ \ +\sup_{x>\kappa_n}\pp\left(\frac{Q_n-\mu_n}{\delta_n}>x\right)
+\sup_{x>\kappa_n}\Big(1-\Phi(x)\Big).~\label{k-1}
\end{align}
By Theorem \ref{thm2-1}, (\ref{d-2}) and the inequality $|e^x-1|\le |x|e^{|x|}$, we deduce that
\begin{align}
&\ \ \ \ \nonumber\sup_{|x|\le \kappa_n}\Bigg|\pp\left(\frac{Q_n-\mu_n}{\delta_n}\le x\right)-\Phi(x)\Bigg|\\
&\le \nonumber\sup_{0<x\le
\kappa_n}\Bigg|\pp\left(\frac{Q_n-\mu_n}{\delta_n}>x\right)-\Big(1-\Phi(x)\Big)\Bigg|
+\sup_{-\kappa_n\le x<0}\Bigg|\pp\left(\frac{Q_n-\mu_n}{\delta_n}\le x\right)-\Phi(x)\Bigg|\\
&\le \nonumber\sup_{0<x\le\kappa_n}M\Big(1-\Phi(x)\Big)\Big(x^3(\varepsilon_n+\eta_n)
+x^2\gamma_n|\ln\gamma_n|+(1+x)(\varepsilon_n|\ln\varepsilon_n|+\eta_n|\ln\eta_n|\Big.\\
& \ \ \ \ \nonumber\Big.+\gamma_n|\ln\gamma_n|)\Big)+\sup_{-\kappa_n\le x<0}M\Phi(x)\Big(|x|^3(\varepsilon_n+\eta_n)+x^2\gamma_n|\ln\gamma_n|\Big.\\
& \ \ \ \ \nonumber\Big.
+(1+|x|)(\varepsilon_n|\ln\varepsilon_n|+\eta_n|\ln\eta_n|+\gamma_n|\ln\gamma_n|)\Big)\\
&\le \nonumber\sup_{0<x\le\kappa_n}\frac{M}{1+x}e^{-x^2/2}\Big(x^3(\varepsilon_n+\eta_n)
+x^2\gamma_n|\ln\gamma_n|+(1+x)(\varepsilon_n|\ln\varepsilon_n|+\eta_n|\ln\eta_n|\Big.\\
& \ \ \ \ \nonumber\Big.+\gamma_n|\ln\gamma_n|)\Big)+\sup_{-\kappa_n\le x<0}\frac{M}{1-x}e^{-x^2/2}\Big(|x|^3(\varepsilon_n+\eta_n)+x^2\gamma_n|\ln\gamma_n|\Big.\\
& \ \ \ \ \nonumber\Big.
+(1+|x|)(\varepsilon_n|\ln\varepsilon_n|+\eta_n|\ln\eta_n|+\gamma_n|\ln\gamma_n|)\Big)\\
&\le \nonumber\sup_{0<x\le\kappa_n}\left\{\frac{M}{x}\Big(x^3(\varepsilon_n+\eta_n)
+x^2\gamma_n|\ln\gamma_n|\Big)+M(\varepsilon_n|\ln\varepsilon_n|+\eta_n|\ln\eta_n|
+\gamma_n|\ln\gamma_n|)\Big)\right\}\\
&\ \ \ \ \nonumber +\sup_{-\kappa_n\le x<0}\left\{\frac{M}{|x|}\Big(|x|^3(\varepsilon_n+\eta_n)+x^2\gamma_n|\ln\gamma_n|\Big)
+M\Big(\varepsilon_n|\ln\varepsilon_n|+\eta_n|\ln\eta_n|+\gamma_n|\ln\gamma_n|\Big)\right\}\\
&\le \nonumber\sup_{0<x\le\kappa_n}M\Big(x^2(\varepsilon_n+\eta_n)+x\gamma_n|\ln\gamma_n|+(\varepsilon_n|\ln\varepsilon_n|+\eta_n|\ln\eta_n|+\gamma_n|\ln\gamma_n|)\Big)\\
&\ \ \ \ \nonumber   +\sup_{-\kappa_n\le x<0}M\Big(x^2(\varepsilon_n+\eta_n)+|x|\gamma_n|\ln\gamma_n|+(\varepsilon_n|\ln\varepsilon_n|+\eta_n|\ln\eta_n|+\gamma_n|\ln\gamma_n|)\Big)\\
&\le \nonumber M\Big(\varepsilon_n^{3/2}+\eta_n^{3/2}+\gamma_n^{3/4}|\ln\gamma_n|
+\varepsilon_n|\ln\varepsilon_n|+\eta_n|\ln\eta_n|+\gamma_n|\ln\gamma_n|\Big)\\
&\le \nonumber M\Big(\varepsilon_n|\ln\varepsilon_n|+\eta_n|\ln\eta_n|+\gamma_n|\ln\gamma_n|\Big)\\
&\le M\eta_n|\ln\eta_n|.~\label{k-2}
\end{align}
In addition, from the inequality (\ref{d-2}), it is easy to see that
\beq\label{k-5}
\aligned
\sup_{x>\kappa_n}\Big(1-\Phi(x)\Big)=\sup_{x<-\kappa_n}\Phi(x)=&\Phi(-\kappa_n)=1-\Phi(\kappa_n)\\
&\le \frac{1}{\sqrt{\pi}(1+\kappa_n)}e^{-\kappa_n^2/2}
\le M\eta_n|\ln\eta_n|.
\endaligned
\deq
Using (\ref{k-2}) and (\ref{k-5}), we have
\begin{align}
\sup_{x<-\kappa_n}\pp\left(\frac{Q_n-\mu_n}{\delta_n}\le x\right)
&\nonumber=\pp\left(\frac{Q_n-\mu_n}{\delta_n}\le -\kappa_n\right)\\
&\le M\eta_n|\ln\eta_n|+\Phi(-\kappa_n)\le M\eta_n|\ln\eta_n|.~\label{k-3}
\end{align}
Similarly, it holds that
\beq\label{k-4}
\sup_{x>\kappa_n}\pp\left(\frac{Q_n-\mu_n}{\delta_n}>x\right)
\le M\eta_n|\ln\eta_n|.
\deq
Therefore, combining the inequalities (\ref{k-1})-(\ref{k-4}) together, the desired result is obtained.
\end{proof}

\begin{proof}[\bf Proof of Theorem \ref{thm2-2}]
It follows from (\ref{h1}) that
\begin{align*}
\sum_{1\leq i\neq j\leq n}B_{ij}^2
&=\sum_{1\leq i\neq j\leq n}\left(A_{ij}^2+\frac{k_i^2k_j^2}{4m^2}-2A_{ij}\frac{k_ik_j}{2m}\right)\\
&=\sum_{1\leq i\neq j\leq n}A_{ij}+\sum_{1\leq i\neq j\leq n}\frac{k_i^2k_j^2}{4m^2}
-\sum_{1\leq i\neq j\leq n}2A_{ij}\frac{k_ik_j}{2m}\\
&=2m+\sum_{1\leq i\neq j\leq n}\frac{k_i^2k_j^2}{4m^2}
-\sum_{1\leq i\neq j\leq n}A_{ij}\frac{k_ik_j}{m}.
\end{align*}
Since $Var(p_{c_1}-p_{(2)})=p_{(3)}-p_{(2)}^2$, from (\ref{2-3-1}), (\ref{q1}) and the above equality, we deduce
\begin{align}
\sigma_n^2&=\nonumber\frac{p_{(2)}+p_{(2)}^2-2p_{(3)}}{2m^2}\sum_{1\leq i\neq j\leq n}B_{ij}^2+\frac{p_{(3)}-p_{(2)}^2}{m^2}\sum_{i=1}^nB_{ii}^2\\
&\ge\nonumber\frac{p_{(2)}+p_{(2)}^2-2p_{(3)}}{2m^2}\sum_{1\leq i\neq j\leq n}B_{ij}^2\\
&=\nonumber\frac{p_{(2)}+p_{(2)}^2-2p_{(3)}}{2m^2}\left(2m+\sum_{1\le i\neq j\le n}\frac{k_i^2k_j^2}{4m^2}-\sum_{1\le i\neq j\le n}A_{ij}\frac{k_ik_j}{m}\right)\\
&\ge\nonumber\frac{p_{(2)}+p_{(2)}^2-2p_{(3)}}{m}\left(1-\sum_{1\le i\neq j\le n}A_{ij}\frac{k_ik_j}{2m^2}\right) \\
&\ge\nonumber\frac{p_{(2)}+p_{(2)}^2-2p_{(3)}}{m}\left(1-\frac{1}{2m^2}\sqrt{\sum_{1\le i\neq j\le n}A_{ij}^2}\sqrt{\sum_{1\le i\neq j\le n}k_i^2k_j^2}\right)\\
&\ge\nonumber\frac{p_{(2)}+p_{(2)}^2-2p_{(3)}}{m}
\left(1-\sqrt{\frac{\left(\sum_{i=1}^nk_i^2\right)^2}{2m^3}}\right)\\
&\ge\nonumber\frac{p_{(2)}+p_{(2)}^2-2p_{(3)}}{m}
\left(1-\frac{\sqrt{2}\max_{1\le i\le n}k_i}{\sqrt{m}}\right)\\
&\ge\delta_n^2\left(1-\frac{\sqrt{2}\max_{1\le i\le n}k_i}{\sqrt{m}}\right).~\label{l-3}
\end{align}
Note that
\begin{align}
\Bigg|\ln\frac{\pp\left(\frac{Q_n-\mu_n}{\sigma_n}>x\right)}{1-\Phi(x)}\Bigg|
&=\nonumber\Bigg|\ln\frac{\pp\left(\frac{Q_n-\mu_n}{\delta_n}>x\frac{\sigma_n}{\delta_n}\right)}{1-\Phi(x)}\Bigg|\\
&=\nonumber\Bigg|\ln\left(\frac{\pp\left(\frac{Q_n-\mu_n}{\delta_n}>x\frac{\sigma_n}{\delta_n}\right)}
{1-\Phi\left(x\frac{\sigma_n}{\delta_n}\right)}
\cdot\frac{1-\Phi\left(x\frac{\sigma_n}{\delta_n}\right)}{1-\Phi(x)}\right)\Bigg|\\
&\le \Bigg|\ln\frac{\pp\left(\frac{Q_n-\mu_n}{\delta_n}>x\frac{\sigma_n}{\delta_n}\right)}
{1-\Phi\left(x\frac{\sigma_n}{\delta_n}\right)}\Bigg|
+\Bigg|\ln\frac{1-\Phi\left(x\frac{\sigma_n}{\delta_n}\right)}{1-\Phi(x)}\Bigg|.~\label{oo-3}
\end{align}
From Theorem \ref{thm2-1}, (\ref{l-3}) and (\ref{d-1}), we have
\begin{align}
&\ \ \ \ \nonumber\Bigg|\ln\frac{\pp\left(\frac{Q_n-\mu_n}{\delta_n}>x\frac{\sigma_n}{\delta_n}\right)}
{1-\Phi\left(x\frac{\sigma_n}{\delta_n}\right)}\Bigg|\\
&\le M\Big(x^3(\varepsilon_n+\eta_{n})+x^2\gamma_n|\ln\gamma_n|
+(1+x)(\varepsilon_n|\ln\varepsilon_n|+\eta_n|\ln\eta_n|+\gamma_n|\ln\gamma_n|)\Big)~\label{oo-4}
\end{align}
and
\begin{align}
\Bigg|\ln\frac{1-\Phi\left(x\frac{\sigma_n}{\delta_n}\right)}{1-\Phi(x)}\Bigg|
&\le \nonumber\Bigg|\ln\frac{1-\Phi\left(x\left(1-\frac{\sqrt{2}\max_{1\le i\le n}k_i}{\sqrt{m}}\right)\right)}{1-\Phi(x)}\Bigg|\\
&=\nonumber\Bigg|\ln\exp\left\{\theta\sqrt{2\pi}(1+x)x\frac{\sqrt{2}\max_{1\le i\le n}k_i}{\sqrt{m}}\right\}\Bigg|\\
&\le M(1+x)x\varepsilon_n\le M(1+x)x\eta_n^2. ~\label{oo-5}
\end{align}
Hence, from (\ref{oo-3}), (\ref{oo-4}) and (\ref{oo-5}), we deduce that
\begin{align*}
\Bigg|\ln\frac{\pp\left(\frac{Q_n-\mu_n}{\sigma_n}>x\right)}{1-\Phi(x)}\Bigg|\le M\Big(x^3 \eta_{n}+x^2\eta_n^2|\ln\eta_n|
+(1+x)\eta_n|\ln\eta_n|\Big),
\end{align*}
which implies that (\ref{ooo-1}) holds.
\end{proof}

\begin{proof}[\bf Proof of Corollary \ref{cor2-21}]
The proof of Corollary \ref{cor2-21} is similar to the proof of Corollary \ref{cor2-2}, so the proof can be omitted.
\end{proof}

\end{document}